\documentclass{article}

\usepackage{geometry}
\usepackage{graphicx, float}
\usepackage{epstopdf}
\usepackage{graphicx}
\usepackage{fancyhdr}
\usepackage{amsmath}
\usepackage{amsthm}
\usepackage{amssymb}
\usepackage{hyperref}
\usepackage{makeidx}
\usepackage{mathdesign}
\usepackage{color}
\usepackage{mathtools}
\definecolor{purple}{rgb}{.9,0,.9}

\def\XXint#1#2#3{{\setbox0=\hbox{$#1{#2#3}{\int}$}
    \vcenter{\hbox{$#2#3$}}\kern-.5\wd0}}

\def\Nat{{\, \hbox{N \kern-1.25em I}\ \;}}
\def\Real{{\, \hbox{R \kern-1.25em I}\ \;}}
\def\Pos{{\, \hbox{P \kern-1.15em I}\ \;}}

\def\Int{{\, \hbox{\tenss Z \kern-1.1em Z} \,}}

\def\rfa{\qquad {\rm for \ all}\ }
\def\rfaa{\qquad {\rm for \ almost \ all}\ }

\font\teneuf=eufm7 at 11pt
\font\seveneuf=eufm7
\font\fiveeuf=eufm5
\def\cA{{\cal A}}\def\cC{{\cal C}} \def\cI{{\cal I}}\def\cK{{\cal K}}\def\cM{{\cal M}}\def\cS{{\cal S}}
\def\cW{{\cal W}}\def\cY{{\cal Y}}

\def\ba{{\bf a}}
\def\bg{{\bf g}}

\def\bn{{\bf n}}
\def\br{{\bf r}}\def\bt{{\bf t}}
\def\bu{{\bf u}}\def\bv{{\bf v}}\def\bx{{\bf x}}
\def\by{{\bf y}}\def\bz{{\bf z}}

\def\bL{{\bf L}}

\def\b0{{\bf 0}}

\newtheorem{theorem}{Theorem}[section]
\newtheorem{lemma}[theorem]{Lemma}

\newtheorem{proposition}[theorem]{Proposition}

\def\tr{\text{tr}}

\def\beqn{\begin{equation}}
\def\eeqn{\end{equation}}

\def\Nat{\mathbb{N}}

\def\Real{\mathbb{R}}
\def\Pos{\mathbb{P}}

\numberwithin{equation}{section}

\def\Int{{\, \hbox{\tenss Z \kern-1.1em Z} \,}}

\def\bnu{{\boldsymbol \nu}}

\newcommand{\half}{\frac{1}{2}}

\makeindex

\newcommand{\abs}[1]{|#1\rvert}
\newcommand{\norm}[1]{\|#1 \|}

\newfont{\ssmsam}{msam5 scaled 750}

\newcommand{\dl}{\,\text{dl}}
\newcommand{\da}{\,\text{da}}

\newcommand{\trans}{\mskip-2mu\scriptscriptstyle\top\mskip-2mu}

\newcommand{\captionfonts}{\footnotesize}

\makeatletter % Allow the use of @ in command names
\long\def\@makecaption#1#2{%
 \vskip\abovecaptionskip
 \sbox\@tempboxa{{\captionfonts #1: #2}}%
 \ifdim \wd\@tempboxa >\hsize
  {\captionfonts #1: #2\par}
 \else
  \hbox to\hsize{\hfil\box\@tempboxa\hfil}%
 \fi
 \vskip\belowcaptionskip}
\makeatother  % Cancel the effect of \makeatletter

\title{Stable and unstable helices:\\ Soap films in cylindrical tubes}
\author{Brian Seguin \& Eliot Fried}

\begin{document}
\date{}
\maketitle

%\tableofcontents

\vspace{.2in}

\begin{abstract}
\noindent Cox \& Jones recently devised and studied an interesting variant of the classical Plateau problem, a variant in which a helical soap film is confined to a cylindrical tube with circular cross-section. Through experiments, numerics, and some analysis, they found that the length and (inner) radius of the tube strongly influence the equilibrium shape of the confined soap film. In this paper, an area minimization problem associated with determining the shape of the film is formulated and analyzed to determine which surfaces are local minima. The connection between a functional inequality and the associated eigenvalue problem plays an important role in the analysis. For helical films, a more detailed analysis is carried out and stability conditions consistent with the experimental and numerical results of Cox \& Jones are obtained.   \end{abstract}

\section{Introduction}

In 1873, Plateau \cite{JP} published a classical work detailing  experiments involving soap films spanning closed wire loops. He also considered soap films spanning multiple disjoint closed loops and found that, depending on the placement of the loops, there may or may not exist a spanning soap film. One particular example discussed by Plateau involves two circular loops of equal radii. If these loops are coaxial and are sufficiently close together, then there is a spanning soap film in the shape of a catenoid. Increasing the distance between the loops causes the neck of the catenoid to shrink. This continues until the ratio of the distance between the loops and their radius reaches a critical value, at which point the soap film becomes unstable and collapses to two disconnected discs corresponding to the Goldschmidt \cite{G} solution.  Other instabilities involving soap films are described by Weaire et al.~\cite{WVTF}.

Recently, Cox \& Jones \cite{CJ} considered the stability of a soap film confined to a tube. Their setup involves a circular cylindrical tube of (inner) radius $R$ and length $L$, with a rigid wire spanning the diameter at each of its ends. Cox \& Jones considered a soap film whose boundary consisted of the two wires and two curves on the inner wall of the tube. Initially, the case in which the two wires are parallel was considered. In this case, when $L$ is small relative to $R$, the soap film forms a flat surface. However, Cox \& Jones found that increasing $L/R$ above a certain threshold destabilizes the soap film. They also performed this experiment after first rotating one wire relative to the other and found that the ratio $L/R$ at which the instability occurs decreases as the relative rotation of the wires increased. Their results suggest that there might be a maximal angle of rotation above which a helical soap film is unstable. Cox \& Jones obtained their results on the basis of experiments and numerical simulations using the program Surface Evolver. They also considered setups involving two or more crossing wires at the ends of the tube, but here we are only interested in the most simple case described above.

Since the days of Euler, it has been known that the shape of a spanning soap film can be determined by considering a minimization problem. In this problem, the wire loop is modeled as a simple closed curve and the soap film as a surface surface. The soap film spanning the wire is represented by the surface spanning the curve whose (Helmholtz) free-energy is minimal---in an appropriate local sense---among all surfaces that span the curve. Under the assumption of constant surface tension---which is reasonable in most situations---this is equivalent to minimizing the area of the spanning surface. The problem of proving the existence of an area minimizing surface with a prescribed boundary is referred to as Plateau's problem. As mentioned earlier, when several closed loops are considered, there need not exist a spanning soap film. That said, Douglas \cite{D}, Rad\'o \cite{R}, and Courant \cite{C} established sufficient conditions for the existence of a spanning film.

Here, to confirm the results of Cox \& Jones, a linear stability analysis is conducted. For parallel wires, the analysis confirms the critical value of $L/R$ at which the flat soap film becomes unstable obtained previously by Cox \& Jones. The analysis also suggests that the critical value of the ratio $L/R$ at which instability occurs decreases on increasing the angle between the wires. Finally, it is shown that there is a maximum rotation above which there exists no stable soap film in the shape of a helicoid.

In Section 2, a variational statement of the problem is provided and essential notation is introduced.  Section 3 consists of two parts.  In the first part, conditions necessary and sufficient for the first variation of the underlying free-energy functional to vanish are obtained. If those conditions hold for a certain surface, that surface is referred to as a critical point of the free-energy functional. In the second part, a condition that ensures the stability of a critical point is obtained. This condition can be expressed as a functional inequality or as a condition on the eigenvalues of a certain partial-differential eigenvalue problem. In Section 4, the stability of a flat surface spanning parallel wires is analyzed. This analysis is generalized in Section 5 to include the case where the wires need not be parallel.  In Section 6, our results and their relation to those of Cox \& Jones \cite{CJ} are discussed.

\section{Formulation}\label{sectsetup}

Consider a tube of length $L$ with an annular cross-section consisting of two concentric circles, the innermost of which has radius $R$. With a rescaling, it suffices to consider a tube with innermost radius $1$ and dimensionless length 
\beqn\label{dimpar}
\rho:=\frac{L}{R}.
\eeqn
Suppose that at each end of the tube there is a wire the centerline of which is a diameter for the innermost cross-section. Assume that both the tube and the wires are rigid. 

The innermost cross-section of the tube is modeled as a cylindrical tube $\cC$ of length $\rho$ and radius $1$ and the wires are modeled as line segments $\cW_1$ and $\cW_2$ of length $2$ (Figure 1). Denote the union of the wires by $\cW:=\cW_1\cup\cW_2$. Choose a unit vector $\bz$ parallel to the axis of $\cC$ directed from $\cW_1$ to $\cW_2$ and denote by $\br$ the unit normal on $\cC$ pointing away from its axis. Choose an $(x,y,z)$ Cartesian coordinate system with origin at the intersection of $\cW_1$ and the axis of the cylinder and with $\cW_2$ intersecting the axis of the cylinder at the point $(0,0,\rho)$. Furthermore, orient the coordinate system so that the $y$-axis is parallel to $\cW_1$. The angle between $\cW_2$ and the $y$-axis measured in the counter-clockwise direction is denoted by $\theta_0$ and satisfies $0\leq\theta_0< \pi$. 

\begin{figure*}[t]\label{FigNotation}
\begin{center}
\begin{picture}(450,250)
\put(25,10)  {\includegraphics[width=2in, height=3in]{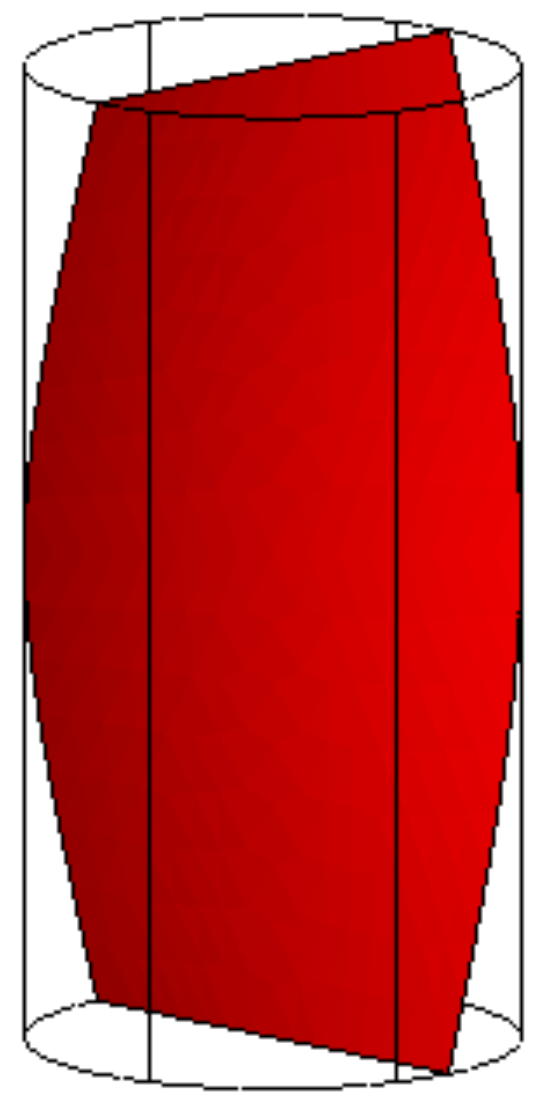}}
\put(245,10){\includegraphics[width=2in, height=3in]{RSVnoM2.pdf}}
\put(90,-20){(a)}
\put(310,-20){(b)}
\put(90,217){$\cW_2$}
\put(90,14){$\cW_1$}
\put(15,115){$\cY_1$}
\put(170,115){$\cY_2$}
\put(400,115){$\rho$}
\put(315,-6){$2$}
\put(387,5){\vector(-1,0){137}}
\put(250,5){\vector(1,0){137}}
\put(395,21){\vector(0,1){192}}
\put(395,213){\vector(0,-1){192}}
\thicklines
\put(317,220){\vector(0,1){25}}
\put(321,235){$\bz$}
\put(306,24){\vector(-3,-2){25}}
\put(302,16){$\bn$}

\put(315,216){\rotatebox[origin=c]{-5}{\vector(-4,-1){25}}}
\put(300,216){$\bt$}
\put(260,170){\vector(-4,-1){25}}
\put(254,175){$\bnu$}
\put(251.5,115){\vector(-3,0){25}}
\put(235,105){$\br$}
\multiput(48,220)(5,-1){17}{\rotatebox[origin=c]{-14}{\line(1,0){3}}}
\put(36,211){$\theta_0$}
\qbezier(56,218.5)(33,211)(59,208)
\end{picture}
\end{center}
\vspace{.2in}
\caption{A surface $\cS$ in $\cA$ confined to the cylinder $\cC$. (a) A depiction of $\cS$ showing that the boundary of $\cS$ consists of two line segments, $\cW_1$ and $\cW_2$, and two curves, $\cY_1$ and $\cY_2$, lying on $\cC$. The line segments $\cW_1$ and $\cW_2$ are diameters of $\cC$ and, thus, are of length 2 and are rotated relative to one another by the angle $\theta_0$. (b) A depiction of $\cS$ showing the unit normal $\bn$ on $\cS$ and the unit tangent $\bt$ and unit binormal $\bnu$ on $\cS$. Here, $\bz$ in the unit vector parallel to the axis of $\cC$ directed from $\cW_1$ to $\cW_2$ and $\br$ is the unit vector on $\cC$ oriented to point away from $\bz$. %The dimensions of the cylinder are also displayed.
}
\end{figure*}

Consider the flat surface 
\beqn
\cS_0:=\{(x,y,z)\in\Real^3\ |\ x=0,\ -1\leq y\leq 1,\ 0\leq z\leq \rho\}.
\eeqn
Put $\cY_0:=\partial\cS_0\setminus\cW$.  Let $\cA$ denote the collection of all surfaces $\cS$ such that
\begin{enumerate}

\item There is an injective mapping $f\in C^2(\cS_0,\Real^3)$ such that $f(\cS_0)=\cS$.

\item The boundary $\partial\cS$ of $\cS$ consists of $\cW$ and two curves that lie on $\cC$---that is, 
\beqn\label{scond}
\cW\subset\partial\cS\qquad{\rm and}\qquad \cY:=\partial\cS\setminus\cW\subset\cC.
\eeqn

\end{enumerate}
Since $f$ is differentiable, it follows from items 1 and 2 that
\beqn\label{fbdy}
f(\cW)=\cW\qquad\text{and}\qquad f(\cY_0)=\cY.
\eeqn

Next, consider a surface $\cS\in\cA$. Because of the assumed regularity, $\cS$ is orientable and, hence, there is unit-vector field $\bn$ of class $C^1$ defined on $\cS$. Denote the tangent and binormal of $\partial\cS$ by $\bt$ and $\bnu$, respectively, and assume that the orientation of $\cS$ is chosen so that, on $\cW_1$, $\bt$ points in the positive $y$-direction (Figure~1).

Consider a $C^1$ surface or curve $\cM$ and a scalar- or vector-valued field $\phi$ defined on $\cM$. The manifold $\cM$ may be the cylinder $\cC$, an element of $\cA$, or part of the boundary of an element of $\cA$. The $\cM$ gradient $\nabla_\cM\phi$ of $\phi$ is defined by 
\beqn
\nabla_\cM\phi(\bx)=\nabla\bar\phi(\bx)\rfa \bx\in\cM,
\eeqn
where $\bar\phi$ is an extension of $\phi$ to a neighborhood of $\cM$ that is constant in the direction perpendicular to $\cM$. If $\phi$ is vector-valued, the $\cM$ divergence $\text{div}_\cM\phi$ of $\phi$ is defined by
\beqn
\text{div}_\cM\phi:=\text{tr}(\nabla_\cM\phi).
\eeqn
The surface Laplacian $\Delta_\cS$ is given by the surface divergence of the surface gradient.  Given a vector field $\bu$ tangent to $\cM$, define the directional derivative $\nabla_\bu\phi$ of $\phi$ in the direction $\bu$ by
\beqn\label{dirder}
\nabla_\bu\phi:=(\nabla_\cM\phi)\bu.
\eeqn
If $\phi$ is scalar-valued, the natural inner-product of $\Real^3$ can be used to write the right-hand side of \eqref{dirder} as
\beqn
(\nabla_\cM\phi)\bu=\nabla_\cM\phi\cdot\bu.
\eeqn

The curvature tensor $\bL$ of a surface $\cS$ in $\cA$ is defined by 
\beqn\label{defL}
\bL:=-\nabla_\cS\bn
\eeqn
and can be shown to be symmetric. The mean and Gaussian curvatures $H$ and $K$ of $\cS$ are given by
\beqn\label{curvatures}
H:=\half \tr \bL\qquad{\rm and}\qquad K:=\half[(\tr \bL )^2-\tr(\bL^2)],
\eeqn
respectively. Since $\cS$ is parameterized by a mapping of class $C^2$, the curvature tensor $\bL$, as well as the mean and Gaussian curvatures $H$ and $K$, are continuous, and hence bounded, on $\cS$.

The functional $E:\cA\longrightarrow\Real$ defined as
\beqn\label{energy}
E(\cS):=\int_\cS\sigma \,\text{da}\rfa \cS\in\cA,
\eeqn
where $\sigma>0$ is the constant, dimensionless surface-tension of the soap film, gives the (Helmholtz) free-energy associated with each surface in $\cA$. The extent to which the local minimizers of $E$ depend on the parameters $\rho$ and $\theta_0$ is considered in the following sections.

%%%%%%%%%%%%%%%%%%

\section{General information about the critical points and their stability}

To analyze the possible critical points of $E$ and their stability, the first and second variations of $E$ are required. For the remainder of this section, consider a fixed surface $\cS\in\cA$ and a one-parameter family of surfaces in $\cA$ generated by a function $f\in C^2(\cI,C^2(\cS_0,\Real^3))$, where $\cI$ is an open interval of $\Real$ containing $0$, such that $\cS=f(0,\cS_0)$ and $f(t,\cdot)\in C^2(\cS_0,\Real^3)$ is injective for all $t\in \cI$. Analogous to \eqref{fbdy}, it follows that
\beqn\label{fbdym}
f(t,\cW)=\cW\qquad\text{and}\qquad f(t,\cY_0)=\cY\rfa t\in \cI.
\eeqn
The initial velocity $\bv$ and acceleration $\ba$ associated with this family are defined by
\beqn\label{velacc}
\bv(\by):=f'(0,\bx)\qquad{\rm and}\qquad\ba(\by):=f''(0,\bx)\rfa \by=f(\bx)\in\cS,
\eeqn
where a prime is used to denote differentiation with respect to $t$.

Equation \eqref{fbdym}$_1$ encompasses the requirement that points on $\cW$ remain on $\cW$, which consists of straight line segments. It follows that
\beqn\label{wcond}
\text{$\bv=(\bv\cdot\bt)\bt\qquad$ and $\qquad\ba=(\ba\cdot\bt)\bt\qquad$ on $\cW$.}
\eeqn
The constraint \eqref{scond}$_2$ can be expressed as $f(t,\bx)\in\cC$ for all $t\in \cI$ and $\bx\in\cY_0$. This immediately implies that
\beqn\label{dscond2}
f'(t,\bx)\cdot\br(f(t,\bx))=0\rfa (t,\bx)\in\cI\times\cY_0.
\eeqn
Evaluating \eqref{dscond2} at $t=0$ yields
\beqn\label{vcond}
\bv\cdot\br=0\qquad {\rm on}\ \cY.
\eeqn
Evaluating the consequence of differentiating \eqref{dscond2} with respect to $t$ at $t=0$ yields
\beqn\label{acond}
\ba\cdot\br=\abs{\bv\cdot\bz}^2-\abs{\bv}^2\leq 0\qquad {\rm on}\ \cY.
\eeqn
To arrive at \eqref{acond} requires the formula 
\beqn\label{cylL}
\nabla_\cC\br=(\bz\times\br)\otimes(\bz\times\br).
\eeqn

It is useful to define
\begin{align}
V&:=\{\bv\in C^2(\cS,\Real^3)\ |\ \text{\eqref{wcond}$_1$ and \eqref{vcond} hold}\}.
\end{align}

\subsection{Characterization of critical points}

To characterize the critical points of $E$, consider the first variation,
\begin{align}
\label{fvar}E(f(\cdot,\cS_0))'|_{t=0}&=\int_\cS \sigma\text{div}_\cS\bv\da\\
\label{stfvar}&=-\int_\cS 2\sigma H(\bv\cdot\bn)\da+\int_{\partial \cS} \sigma(\bv\cdot\bnu)\dl,
\end{align}
of the surface free-energy functional $E$ at $\cS$; \eqref{stfvar} arises from \eqref{fvar} upon using the surface-divergence theorem, which appears on page 222 in the book by Brand \cite{B}. So far, neither \eqref{wcond}$_1$ nor \eqref{vcond} has been used.

The following result characterizes the critical points of $E$.
\begin{proposition}\label{eqcp}
The surface $\cS$ is a critical point of the functional $E$ if and only if
\begin{align}
\label{H0}H&=0\qquad\text{on}\ \cS,
\\[2pt]
\label{nuc}\bnu&=\br\qquad\text{on}\ \cY.
\end{align}
\end{proposition}

\begin{proof}
By definition, $\cS$ is a critical point of $E$ if and only if the first variation with respect to every admissible velocity $\bv$ vanishes---that is, if and only if
\beqn\label{fvcond}
0=-\int_\cS 2H(\bv\cdot\bn)\da+\int_{\partial \cS} \bv\cdot\bnu\dl\rfa \bv\in V.
\eeqn

Assume that \eqref{fvcond} holds. Since the normal component of a vector field belong to $V$ is unconstrained, \eqref{fvcond} implies that \eqref{H0} holds. On invoking \eqref{wcond}$_1$ and \eqref{H0}, \eqref{fvcond} reduces to
\beqn\label{rfirstvar}
0=\int_{\cY}\bv\cdot\bnu\, \text{dl}\rfa \bv\in V.
\eeqn
Since $\{\br,\bt,\bg\}$, with $\bg:=\br\times\bt$, is an orthonormal basis on $\cY$ and \eqref{vcond} holds, $\bv$ admits a decomposition of the form
\beqn\label{vrep}
\bv=(\bv\cdot\bt)\bt+(\bv\cdot\bg)\bg\qquad\text{on $\cY$,\hspace{.1in} for all $\bv\in V.$}
\eeqn
On using \eqref{vrep} and the fact that $\bnu\cdot\bt=0$ on $\cY$, \eqref{rfirstvar} becomes
\beqn\label{evar}
0=\int_\cY (\bv\cdot\bg) (\bg\cdot\bnu)\, \text{dl}\rfa \bv\in V.
\eeqn
Since the component of vector field in $V$ in the direction $\bg$ is unrestricted, \eqref{evar} yields $\bg\cdot\bnu=0$ and, hence, \eqref{nuc} holds.

It is easy shown that the conditions \eqref{H0} and \eqref{nuc} imply \eqref{fvcond}.
\end{proof}

As expected, \eqref{H0} is the classical condition that the mean curvature of a critical point of $E$ must vanish. The condition \eqref{nuc} embodies the requirement that the portion $\cY$ of its boundary $\partial\cS$ disjoint from $\cW$ meets the inner wall of $\cC$ at a right angle.

It is possible to construct a family of surfaces that satisfy the two conditions of Proposition \ref{eqcp}. Fix an integer $n$ and put $\theta:=\theta_0+\pi n$. Consider the helicoid $\cS_\theta$ with the parameterization
\beqn\label{helicoid}
f(y,z)=(-y\sin(\theta z/\rho),y\cos(\theta z/\rho),z)\rfa y\in [-1,1],\ z\in[0,\rho].
\eeqn
The integer $n$ indicates how many half twists the helicoid makes. The boundary $\partial\cS_\theta$ of $\cS_\theta$ satisfies the conditions \eqref{scond}. It is not difficult to see that \eqref{nuc} holds for all such surfaces; moreover, it is well-known that the mean curvature of a helicoid vanishes. Hence, $\cS_\theta$ is a critical point of $E$. It seems reasonable to conjecture that these helicoids are the only surfaces in $\cA$ that satisfy the conditions \eqref{H0} and \eqref{nuc} and the boundary condition \eqref{scond}.

\subsection{Stability of critical points}

The second variation of $E$, while less well-known, is provided by Simon \cite{Simon} and takes the form
\beqn\label{svar}
E(f(\cdot,\cS_0))''|_{t=0}=\int_\cS\sigma\big[\abs{(\nabla_\cS\bv) ^{\trans}\bn}^2+(\text{div}_\cS\bv)^2-\text{tr}(\nabla_\cS\bv^2)+\text{div}_\cS\ba\big]\da.
\eeqn
A lengthy but straightforward calculation shows that 
\begin{align}
\nonumber(\text{div}_\cS\bv)^2-\text{tr}(\nabla_\cS\bv^2)&=\bv\cdot\big[\text{div}_\cS[(\nabla_\cS\bv) ^{\trans}]-\nabla_\cS(\text{div}_\cS\bv)]+\text{div}_\cS[(\text{div}_\cS\bv)\bv-(\nabla_\cS\bv)\bv]
\\[2pt]
\label{oddrel}&=(\bL\cdot\nabla_\cS\bv)(\bv\cdot\bn)-\bL\bv\cdot(\nabla_\cS\bv) ^{\trans}\bn+\text{div}_\cS[(\text{div}_\cS\bv)\bv-(\nabla_\cS\bv)\bv].
\end{align}
Using \eqref{oddrel} in \eqref{svar} and applying the surface-divergence theorem yields
\begin{align}
\nonumber E(f(\cdot,\cS_0))''|_{t=0}&=\int_\cS\sigma\big[\abs{(\nabla_\cS\bv)^{\trans}\bn}^2+(\bL\cdot\nabla_\cS\bv)(\bv\cdot\bn)-\bL\bv\cdot(\nabla_\cS\bv) ^{\trans}\bn\\
&\qquad\qquad+H[(\text{div}_\cS\bv)\bv-(\nabla_\cS\bv)\bv+\ba]\cdot\bn\big]\da\\
\label{stsvar}&\qquad +\int_{\partial\cS}\sigma[(\text{div}_\cS\bv)\bv-(\nabla_\cS\bv)\bv+\ba]\cdot\bnu\dl.
\end{align}
By definition, the surface $\cS$ satisfies the second-variation stability condition if
\beqn\label{svsc}
E(f(\cdot,\cS_0))''|_{t=0}\geq 0\rfa \bv\in V.
\eeqn
Although the second-variation stability condition is necessary for the surface to be a local minimum of $E$, it is not sufficient.

To characterize the critical points of $E$ that satisfy the second-variation stability condition, it is useful to introduce
\beqn\label{normvar}
V_N:=\{c\in C^2(\cS,\Real)\ |\ c=0\ {\rm on}\ \cW\}.
\eeqn
In proving the next proposition, it will become clear that the elements of $V_N$ represent the normal components of elements of $V$.

\begin{proposition}\label{propstcp}
A critical surface $\cS$ satisfies the second-variation stability condition if and only if
\beqn\label{stsvareq}
\int_\cY (\bt\cdot\bz)^2c^2{\rm dl}\leq\int_\cS( \abs{\nabla_\cS c}^2+2Kc^2)\, {\rm da}\rfa c\in V_N.\\
\eeqn

\end{proposition}

\begin{proof}
Suppose now that $\cS$ is a critical point of $E$. First, using \eqref{wcond}$_1$, it follows that
\beqn\label{wcond0}
[(\nabla_\cS\bv)\bv]\cdot\bnu=(\bv\cdot\bt)[\nabla_\bt((\bv\cdot\bt)\bt)]\cdot\bnu=(\bv\cdot\bt)[\nabla_\bt(\bv\cdot\bt)](\bt\cdot\bnu)=0\qquad{\rm on}\ \cW,
\eeqn
Using \eqref{wcond}, \eqref{acond}, Proposition \ref{eqcp}, and \eqref{wcond0} in \eqref{stsvar} yields
\begin{align}
\nonumber E(f(\cdot,\cS_0))''|_{t=0}&=\int_\cS\sigma\big[\abs{(\nabla_\cS\bv) ^{\trans}\bn}^2+(\bL\cdot\nabla_\cS\bv)(\bv\cdot\bn)-\bL\bv\cdot(\nabla_\cS\bv) ^{\trans}\bn\big]\da\\
\label{stsvar2}&\qquad -\int_{\cY}\sigma\big([(\nabla_\cS\bv)\bv]\cdot\br+\abs{\bv}^2-\abs{\bv\cdot\bz}^2\big)\dl.
\end{align}
Hence, $\cS$ satisfies the second variation stability condition if and only if
\begin{align}
\nonumber0&\leq\int_\cS\big[\abs{(\nabla_\cS\bv) ^{\trans}\bn}^2+(\bL\cdot\nabla_\cS\bv)(\bv\cdot\bn)-\bL\bv\cdot(\nabla_\cS\bv) ^{\trans}\bn\big]\da\\
\label{stcond}&\qquad -\int_{\cY}\big([(\nabla_\cS\bv)\bv]\cdot\br+\abs{\bv}^2-\abs{\bv\cdot\bz}^2\big)\dl\rfa \bv\in V.
%-\int_\cW((\nabla_\cS\bv)\bv)\cdot\br\dl\rfa \bv\in V.
\end{align}

To analyze \eqref{stcond}, it is useful to decompose each $\bv\in V$ into tangential and normal components. In particular, for all $\bv\in V$, there is a vector field $\bu$ and a scalar field $c$ defined on $\cS$, both of class $C^2$, such that
\beqn\label{vdecomp}
\bv=\bu+c\bn\qquad{\rm with}\qquad \bu\cdot\bn=0\qquad \text{on}\ \cS
\eeqn
where, from \eqref{wcond} and \eqref{vcond}, $\bu$ and $c$ satisfy
\begin{align}
\label{decompcond1}\bu=(\bu\cdot\bt)\bt\qquad \text{and} \qquad c=0\qquad& \text{on } \cW,\\[2pt]
\label{decompcond2}\bu\cdot\br=0\hspace{.75in}& \text{on}\ \cY.
\end{align}
Applying the surface gradient to the relation $\bu\cdot\bn=0$, which holds on $\cS$, and using \eqref{defL} yields
\beqn\label{tancurve}
(\nabla_\cS\bu) ^{\trans}\bn=\bL\bu\qquad \text{on}\ \cS.
\eeqn

Fix an arbitrary $\bv\in V$, and hence a corresponding $\bu$ and $c$. Using \eqref{curvatures}$_2$, \eqref{vdecomp}, and \eqref{tancurve}, straightforward calculations show that
\beqn\label{decomp1}
\left.
\begin{array}{rl}
\abs{(\nabla_\cS\bv) ^{\trans}\bn}^2=&\abs{\bL\bu}^2+2\bL\bu\cdot\nabla_\cS c+\abs{\nabla_\cS c}^2,
\\[6pt]
\ (\bL\cdot\nabla_\cS\bv)(\bv\cdot\bn)=&(\bL\cdot\nabla_\cS\bu)c+2Kc^2,\\[6pt]
\bL\bv\cdot(\nabla_\cS\bv) ^{\trans}\bn=&\abs{\bL\bu}^2+\bL\bu\cdot\nabla_\cS c,
\end{array}
\right\}\qquad\text{on }\cS.
\eeqn
Moreover, using \eqref{vdecomp} and \eqref{decompcond2} yields
\beqn\label{decomp2}
\left.
\begin{array}{rl}
[(\nabla_\cS\bv)\bv]\cdot\br=&(\bu\cdot\bt)^2[(\nabla_\bt\bt)\cdot\br]+(\bu\cdot\bt)[(\nabla_\bt\bn)\cdot\br]c,
\\[6pt]
\abs{\bv}^2-\abs{\bv\cdot\bz}^2=&(\bt\cdot\bz)^2c^2+(\bu\cdot\bt)^2(\bn\cdot\bz)^2-2(\bu\cdot\bt)(\bn\cdot\bz)(\bt\cdot\bz)c,
\end{array}
\right\}\qquad\text{on }\cY.
\eeqn
In view of \eqref{cylL} and the identities $\bt=\bn\times\bnu$ and $\bn=\bnu\times\bt$ connecting $\bt$, $\bn$, and $\bnu$, it follows that
\beqn\label{decomp3}
\left.
\begin{array}{rl}
\nabla_\bt\bt\cdot\br=&-(\bn\cdot\bz)^2,
\\[6pt]
\nabla_\bt\bn\cdot\br=&(\bn\cdot\bz)(\bt\cdot\bz),
\end{array}
\right\}\qquad\text{on }\cY.
\eeqn
With \eqref{vdecomp}--\eqref{decomp3}, \eqref{stcond} becomes
\begin{align}
\label{stcond2}0&\leq\int_\cS\big[ \abs{\nabla_\cS c}^2+2Kc^2 +\bL\cdot(c\nabla_\cS\bu+\nabla_\cS c\otimes\bu)\big]\da -\int_{\cY}\big[ (\bt\cdot\bz)^2c^2-(\bu\cdot\bt)(\bn\cdot\bz)(\bt\cdot\bz)c \big]\dl.
\end{align}
Notice that by the surface-divergence theorem, integration by parts, the identity $\text{div}_\cS\bL=0$, which follows from \eqref{H0}, and \eqref{decompcond2},
\beqn\label{decomp4}
\int_\cS\bL\cdot(c\nabla_\cS\bu+\nabla_\cS c\otimes\bu)\da=\int_\cS\bL\cdot\nabla_\cS(c\bu)=\int_\cY c\bL\bu\cdot\br=\int_\cY-c(\bu\cdot\bt)\nabla_\bt\bn\cdot\br,
\eeqn
which, with \eqref{decomp3}$_2$, can be used to reduce \eqref{stcond2} to
\beqn\label{stcond3}
\int_\cY (\bt\cdot\bz)^2c^2\dl\leq\int_\cS( \abs{\nabla_\cS c}^2+2Kc^2)\da.
\eeqn

Since $\bv$ is an arbitrary element of $V$, the $c$ appearing in \eqref{stcond2} can be chosen arbitrarily as long as it is consistent with the condition $c=0$ on $\cW$. Hence, \eqref{stsvareq} holds. Showing that \eqref{stcond3} implies \eqref{stcond} is a matter of using the identities \eqref{decomp1}--\eqref{decomp3} and \eqref{decomp4} and tracing the proof backwards.
\end{proof}

Proposition \ref{propstcp} gives an alternate statement of the second-variation stability condition in terms of a functional inequality.  This inequality resembles a Sobolev trace inequality involving the $L^2$ norm on the boundary weighted with $(\bt\cdot\bz)^2$ and part of the $W^{1,2}$ norm on $\cS$ weighted with $2K$.  Equation \eqref{stsvareq} expresses the requirement that the optimal constant for the trace inequality with these weighted norms is less than or equal to $1$.  This condition places restrictions on the possible surfaces that may constitute local minima of the free-energy $E$.  Notice that this is in opposition to the usual discussion of optimal constants in Sobolev trace inequalities.  Traditionally, the trace inequality is applied and the optimal constant is sought over a given domain.  Here, a bound on the optimal constant provides information regarding the geometry of the domain.

It is possible to find an equivalent statement of the second-variation stability condition involving a partial-differential eigenvalue problem.

\begin{proposition}\label{propstcp2}
The surface $\cS$ satisfies the second-variation stability condition if and only if the smallest eigenvalue $\lambda$ of the problem
\beqn\label{evp}
\left.
\begin{array}{cl}
\Delta_\cS c-2\lambda Kc=0,&\qquad {\rm on}\ \cS,\\[6pt]
\nabla_\br c-\lambda (\bt\cdot\bz)^2c=0,&\qquad {\rm on}\ \cY,\\[6pt]
c=0,&\qquad{\rm on}\ \cW,
\end{array}
\right\}
\eeqn
is greater than or equal to $1$, where weak solutions to \eqref{evp} are sought in $W^{1,2}(\cS,\Real)$.
\end{proposition}

\begin{proof}
To establish this result, the characterization of the second-variation stability condition given in Proposition \ref{propstcp} will be used. First notice that given $c\in V_N$ for which
\beqn\label{nonzero}
\int_\cY (\bt\cdot\bz)^2c^2\dl-\int_\cS 2Kc^2\da
\eeqn
vanishes, the inequality \eqref{stsvar2} holds trivially. Moreover, since its mean curvature vanishes, the Gaussian curvature of $\cS$ obeys $K\leq0$; hence, \eqref{nonzero} is never negative. It follows that \eqref{stsvareq} is equivalent to
\beqn\label{eqineq}
1\leq \frac{\int_\cS\abs{\nabla_\cS c}^2\da}{\int_\cY (\bt\cdot\bz)^2c^2\dl-\int_\cS 2Kc^2\da}=:J(c)\rfa c\in V'_N,
\eeqn
where $V'_N:=\{c\in V_N\ |\ \eqref{nonzero}\ \text{is nonzero}\}$. 

First, it is shown that $J$, defined in \eqref{eqineq}, has a minimizer in $V_N'$. Toward this end, let $\{c_n\}_{n\in\Nat}$ be a minimizing sequence in $V'_N$. Since scaling the terms of this sequence does not effect its minimizing property, assume, without loss of generality, that the terms of the sequence are normalized according to
\beqn\label{normalize}
\int_\cS (c_n^2+\abs{\nabla_\cS c_n}^2)\da = 1\rfa n\in\Nat.
\eeqn
It follows trivially that the sequence is bounded in $W^{1,2}(\cS,\Real)$ and that, by the Banach--Alaoglu theorem and a Sobolev trace inequality, there is a $c_\infty\in W^{1,2}(\cS,\Real)$ and a subsequence of $\{c_n\}_{n\in\Nat}$, which for simplicity is also denoted by $\{c_n\}_{n\in\Nat}$, such that
\begin{align}
\label{conv1}c_n&\xrightharpoonup{W^{1,2}(\cS,\Real)} c_\infty,\\
\label{conv2}c_n&\xrightarrow{\hspace{.06in} L^2(\cS,\Real)\hspace{.06in}} c_\infty,\\
\label{conv3}c_n&\xrightarrow{\hspace{.025in}L^2(\partial\cS,\Real)\hspace{.025in}} c_\infty.
\end{align}
Next, it must be shown that $c_\infty\in V'_N$ and that $c_\infty$ minimizes $J$. Since $c_n=0$ on $\cW$ for each $n\in\Nat$, it follows from \eqref{conv3} that $c_\infty=0$ on $\cW$. To see that \eqref{nonzero} is nonzero for $c$ replaced by $c_\infty$, begin by assuming that it is zero and seek a contradiction. Since $\{c_n\}_{n\in\Nat}$ is a minimizing sequence,
\beqn\label{Jbound}
\lim_{n\rightarrow\infty}\frac{\int_\cS\abs{\nabla_\cS c_n}^2\da}{\int_\cY (\bt\cdot\bz)^2c_n^2\dl-\int_\cS 2Kc_n^2\da}=\lim_{n\rightarrow\infty}J(c_n)<\infty.
\eeqn
It follows from \eqref{conv2}, \eqref{conv3}, and the boundedness of both $\bt\cdot\bz$ and $K$ that
\beqn\label{limit}
\lim_{n\rightarrow\infty}\int_\cY (\bt\cdot\bz)^2c_n^2\dl-\int_\cS 2Kc_n^2\da=\int_\cY (\bt\cdot\bz)^2c_\infty^2\dl-\int_\cS 2Kc_\infty^2\da.
\eeqn
Since, by assumption, the right-hand side of \eqref{limit} is zero, \eqref{Jbound} and \eqref{limit} together yield
\beqn\label{known}
0=\lim_{n\rightarrow\infty}\int_\cS\abs{\nabla_\cS c_n}^2\da\geq\int_\cS\abs{\nabla_\cS c_\infty}^2\da.
\eeqn
Hence, by the Poincar\'e inequality, $\int_\cS c_\infty^2\da=0$.  This together with \eqref{known} and \eqref{normalize} yields a contradiction from which it follows that \eqref{nonzero} must be nonzero for $c$ replaced by $c_\infty$; thus, $c_\infty\in V'_N$. Thus, by \eqref{conv1}--\eqref{conv3}, $J(c_\infty)$ must satisfy the inequality
\beqn\label{wlsc}
J(c_\infty)\leq\lim_{n\rightarrow\infty}J(c_n).
\eeqn
Since $\{c_n\}_{n\in\Nat}$ is a minimizing sequence, \eqref{wlsc} implies that $c_\infty$ is a minimizer of $J$.

Assume that \eqref{eqineq} holds and that $c\in V'_N$ is a minimizer of $J$. Find a sequence $\{c_n\}_{n\in\Nat}\subset V'_N$ converging to $c$ strongly in $W^{1,2}(\cS,\Real)$. It follows from \eqref{eqineq} that
\beqn
\lambda:=J(c)=\lim_{n\rightarrow \infty}J(c_n)\geq 1.
\eeqn
Since $c$ is a minimizer of $J$, the first variation of $J$ at $c$ must vanish. This can be expressed as
\beqn\label{fvcJ}
\int_\cS (\nabla_\cS c\cdot\nabla_\cS\xi+2\lambda K\hspace{-.02in} c \xi)\da-\int_\cY (\bt\cdot\bz)^2c\xi\dl=0\rfa \xi\in C^\infty(\cS,\Real)\ \text{with}\ \xi=0\ {\rm on}\ \cW.
\eeqn
The condition \eqref{fvcJ} requires that $c$ is a weak solution of \eqref{evp}.

Conversely, assume that $c\in W^{1,2}(\cS,\Real)$ is a weak solution of \eqref{evp}---that is, assume that \eqref{fvcJ} holds and that $c=0$ on $\cW$---with minimal value $\lambda\geq 1$. It follows from \eqref{fvcJ} that $c$ is a critical point of $J$. Since $\lambda$ is minimal and $J$ has a minimizer, $c$ must minimize $J$. Moreover, since $\lambda\geq 1$, \eqref{eqineq} must hold.
\end{proof}

Proposition \ref{propstcp} or, equivalently, Proposition \ref{propstcp2}, places restrictions on stable critical points of the free-energy functional $E$, as will become clear in the next section on considering the surface $\cS_0$. Ideally it should be possible to use one of the last two propositions to obtain explicit information about the stable critical points, such as, for example, a condition on $\rho$ and $\theta$ which ensures that $\cS_0$ is stable.

While the two previous propositions give conditions necessary for stability, they do not give conditions sufficient for stability. However, it is true that if a critical point $\cS$ of $E$ satisfies
\beqn\label{stsvarseq}
\int_\cY (\bt\cdot\bz)^2c^2{\rm dl}<\int_\cS( \abs{\nabla_\cS c}^2+2Kc^2)\, {\rm da}\rfa c\in V_N,\\
\eeqn
then that critical point is stable. Also, if there is a $c\in V_N$ such that
\beqn\label{instsvarseq}
\int_\cY (\bt\cdot\bz)^2c^2{\rm dl}>\int_\cS( \abs{\nabla_\cS c}^2+2Kc^2)\, {\rm da},\\
\eeqn
then that critical point is unstable. Phrased in terms of the eigenvalue problem \eqref{evp}, if the smallest eigenvalue is strictly greater than $1$, then the relevant critical point is stable, but if it is strictly less than $1$, then the relevant critical point is unstable.

%%%%%%%%%%%%%%%%%%%%

\section{Stability analysis of $\cS_0$}

Suppose that $\theta_0=0$, in which case the surface $\cS_0$ described in the second paragraph of Section \ref{sectsetup} belongs to $\cA$. It is immediately obvious that $\cS_0$ satisfies \eqref{H0} and \eqref{nuc} and, hence, is a critical point of $E$. Whether $\cS_0$ satisfies the second-variation stability condition---that is, whether \eqref{stcond} holds---is not, however, obvious. It transpires that $\cS_0$ need not be stable. In this section, conditions on $\rho$ necessary and sufficient to ensure that \eqref{stsvareq} holds are determined.

Since $\cS_0$ is flat, its curvature vanishes and, hence, \eqref{stsvareq} reduces to
\beqn\label{svcond0}
\int_{\cY_0}c^2\dl\leq\int_{\cS_0}\abs{\nabla_{\cS_0} c}^2\da\rfa c\in V_N;
\eeqn
further, the eigenvalue problem \eqref{evp} reduces to
\beqn\label{evp0}
\left.
\begin{array}{cl}
\Delta c=0,&\qquad {\rm on}\ [-1,1]\times [0,\rho],\\[6pt]
c_y(1,z)-\lambda c(1,z)=0&\rfa z\in [0,\rho],\\[6pt]
c_y(-1,z)+\lambda c(-1,z)=0&\rfa z\in [0,\rho],\\[6pt]
c(y,0)=c(y,\rho)=0,&\rfa y\in[-1,1].
\end{array}
\right\}
\eeqn

\begin{proposition}\label{prop4.1}
The surface $\cS_0$ satisfies the second-variation stability condition if and only if 
\beqn\label{Sostc}
1\leq \frac{\pi}{\rho}\tanh\big(\frac{\pi}{\rho}\big).
\eeqn
\end{proposition}

\begin{proof}
To demonstrate the necessity of \eqref{Sostc}, assume that \eqref{svcond0} holds and consider the function $c$ defined by
\beqn\label{spc}
c(y,z):=\cosh(\pi y/\rho)\sin\big(\pi z/\rho)\rfa (y,z)\in[-1,1]\times[0,\rho].
\eeqn
Notice that, for this choice of $c$, \eqref{decompcond1} is satisfied and \eqref{svcond0} yields
\begin{align}
0&\leq\int_{-1}^1\int_0^{\rho}(c_y^2(y,z)+c_z^2(z))\,\text{d}z\text{d}y-\int_0^L(c^2(-1,z)+c^2(1,z))\,\text{dl}\\
&\leq \pi\sinh\big(\pi/\rho)\cosh(\pi/\rho)-\rho\cosh^2\big(\pi/\rho),
\end{align}
which is equivalent to \eqref{Sostc}.

Assume now that \eqref{Sostc} holds. It will be shown that the minimum value of $\lambda$ allowed by the eigenvalue problem \eqref{evp0} must be greater than or equal to $1$. Let $c\in W^{1,2}(\cS_0,\Real)$ be a weak solution to \eqref{evp0} with eigenvalue $\lambda$. Notice that because $c=0$ on $\cW$, $c$ admits a sine-series representation of the form
\beqn\label{cexp}
c(y,z)=\sum_{k\in\Nat}g_k(y)\sin(k\pi z/\rho)\qquad\text{for almost all}\ (y,z)\in[-1,1]\times[0,\rho],
\eeqn
where the functions $g_k$, $k\in\Nat$, are of class $W^{1,2}$ on $\cS$ and the sum converges in the associated norm.  Substituting the expansion \eqref{cexp} into \eqref{evp0} yields
\begin{align}
 \sum_{k\in\Nat}\big[g_k''(y)-\frac{k^2\pi^2}{\rho^2}g_k(y)\big]\sin(k\pi z/\rho)=0
&\qquad \text{for almost all}\ (y,z)\in[-1,1]\times[0,\rho],\\
\sum_{k\in\Nat}\big[g_k'(1)-\lambda g_k(1)\big]\sin(k\pi z\rho)=0&\rfa z\in [0,\rho],\\
\sum_{k\in\Nat}\big[g_k'(-1)+\lambda g_k(-1)\big]\sin(k\pi z/\rho)=0&\rfa z\in [0,\rho],
\end{align}
Since the family of functions $z\mapsto\sin(k\pi z/\rho)$, $k\in\Nat$, is a complete orthogonal basis for functions in $W^{1,2}([0,\rho],\Real)$ with zero Dirichlet boundary condition, for each $k\in\Nat$, $g_k$ must satisfy
\beqn\label{odemin}
\left.
\begin{array}{c}
g''_k-\frac{k^2\pi^2}{\rho^2}g_k=0\qquad\text{on}\ [-1,1],\\[6pt]
g'_k(1)-\lambda g_k(1)=0,\\[6pt]
g'_k(-1)+\lambda g_k(-1)=0.
\end{array}
\right \}
\eeqn

It is easy to show that any solution of \eqref{odemin}$_1$ must be a linear combination of hyperbolic sines and cosines and that the function that delivers the minimal value of $\lambda$ that satisfies \eqref{odemin} arises for $k=1$ and is given by
\beqn\label{g1}
g_1(y)=\cosh(\pi y/\rho)\rfa y\in[-1,1].
\eeqn
The value of $\lambda$ corresponding to $g_1$ is
\beqn
\lambda=\frac{\pi}{\rho}\tanh\big(\frac{\pi}{\rho}\big).
\eeqn
Since \eqref{Sostc} holds, it follows immediately that $\lambda \geq 1$. Moreover, since this $\lambda $ is the smallest eigenvalue of the eigenvalue problem \eqref{odemin}, Proposition \ref{propstcp2} ensures that $\cS_0$ satisfies the second-variation stability condition.
\end{proof}

The previous result embodies the requirement that $\cS_0$ satisfies the second-variation stability condition when $\rho$ satisfies \eqref{Sostc}, which can be approximated as $\rho\lessapprox 2.62$. Hence, for $2.62\lessapprox \rho$, $\cS_0$ is not a local minimum of $E$. This agrees with the results of Cox \& Jones \cite{CJ}. It should be noted that Cox \& Jones did obtain this critical value of $\rho$ by using a linear stability analysis but, as they themselves admit, they did not accurately account for the boundary condition \eqref{nuc} when they considered a perturbation of the flat surface.

It ensues that for $\pi/2<\rho$, $\cS_0$ is not a global minimizer. To see this, fix a $\phi\in(0,\pi/2)$ and consider the curve $\gamma:[0, 2\phi+L]\longrightarrow\Real^3$ given by
\beqn
\gamma(s):=\left\{
\begin{array}{cl}
(\sin s,\cos s ,0)& s\in[0,\phi),\\[4pt]
(\sin \phi ,\cos \phi ,s-\phi)&s\in[\phi,\phi+\rho),\\[4pt]
(\sin(2\phi+\rho-s),\cos(2\phi+\rho-s), \rho)&s\in[\phi+\rho,2\phi+\rho].
\end{array}
\right.
\eeqn
Let $\cY_1$ be the range of $\gamma$ and $\cY_2$ be the reflection of $\cY_1$ across the $y=0$ plane. Consider the surface $\cS_\phi$, with boundary $\cW\cup\cY_1\cup\cY_2$, consisting of three flat subsurfaces, two of which are part of discs and lie in the planes $z=0$ and $z=\rho$ and the other of which is rectangular and lies in the plane $x=\sin \phi$ (Figure 2). Since this surface is not of class $C^2$, it does not belong to $\cA$. However, it is possible to find an element of $\cA$ as close to that of $\cS_\phi$ as desired. Computing the free-energy of $\cS_\phi$ yields
\beqn
E(\cS_\phi)=\rho \cos\phi+2\phi;
\eeqn
hence, for $\rho>\pi/2$, there is a $\phi$ close to $\pi/2$ such that
\beqn
E(\cS_\phi)<E(\cS_0).
\eeqn
Thus, $\cS_0$ is not a global minimizer for $\rho>\pi/2$. 

\begin{figure*}[t]\label{localminnh}
\begin{center}
\begin{picture}(450,250)
\put(25,10)  {\includegraphics[width=2in, height=3in]{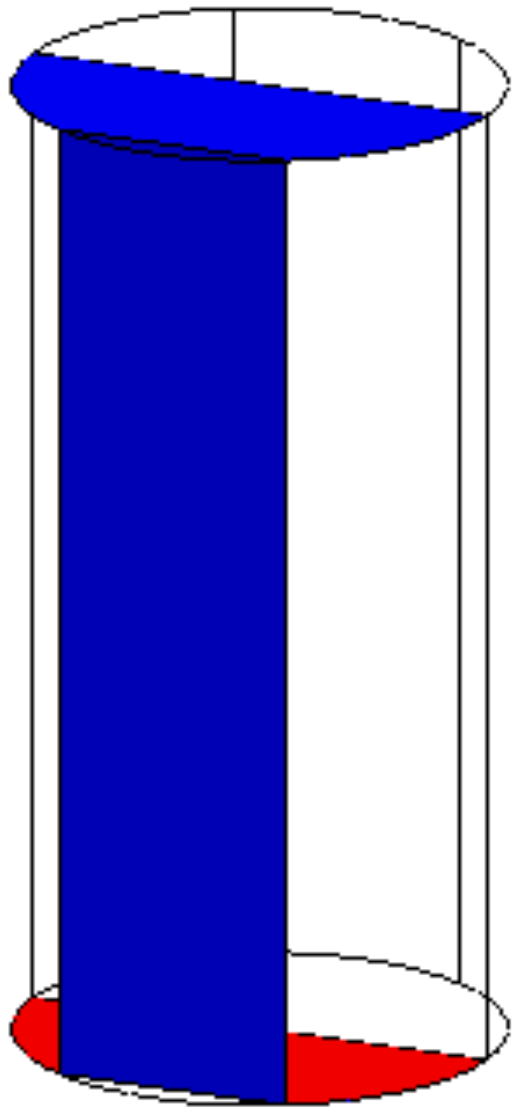}}
\put(245,10){\includegraphics[width=2in, height=3in]{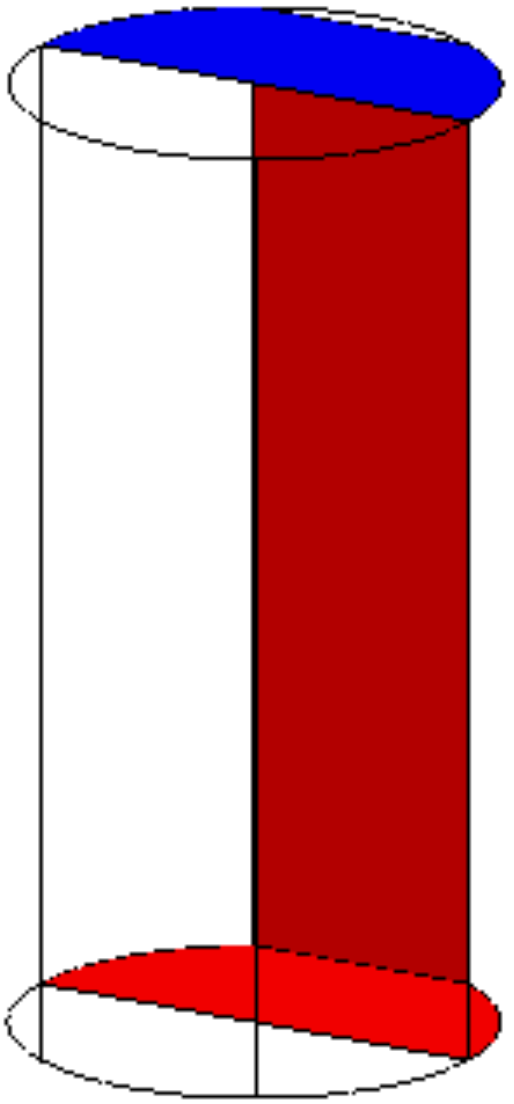}}
\put(90,-5){(a)}
\put(310,-5){(b)}
\end{picture}
\end{center}
\caption{The surface $\cS_\phi$ depicted from two different angles. Red and blue are used to indicate different sides of the surface.} 
\end{figure*}

%%%%%%%%%%%%%%%%%%%%

\section{Stability analysis of $\cS_\theta$}

This section returns to the general case $\theta_0\not=0$ and focuses on the question of whether the helicoid $\cS_\theta$ parameterized by \eqref{helicoid} satisfies the second-variation stability condition. Recall that $\theta=\theta_0+\pi n$ for some $n\in\Nat$. Due to mirror symmetry, it suffices to consider the case $\theta\geq0$. As will become evident, it is the value of $\theta$ rather than those of $\theta_0$ and $n$ that emerge naturally in the analysis. For this reason, dependence on $n$ is not indicated explicitly.

Using the parameterization \eqref{helicoid}, various calculations show that
\begin{align}
\label{parrel1}\Delta_{\cS_\theta}c(y,z)&=c_{yy}(y,z)+\frac{\theta^2 y}{\rho^2+\theta^2y^2}c_y(y,z)+\frac{\rho^2}{\rho^2+\theta^2y^2}c_{zz}(y,z),\\
K(y,z)&=-\frac{\rho^2\theta^2}{(\rho^2+\theta^2y^2)^2},\\
(\bt(z)\cdot\bz)^2&=\frac{\rho^2}{\rho^2+\theta^2},\\[4pt]
\nabla_\br c(1,z)&=c_y(1,z),\\[4pt]
\label{parrel5}\nabla_\br c(-1,z)&=-c_y(-1,z)
\end{align}
for all $y\in[-1,1]$ and $z\in[0,\rho]$. On using the relations \eqref{parrel1}--\eqref{parrel5}, the eigenvalue problem \eqref{evp} can be written as
\beqn\label{hevp}
\left.
\begin{array}{c}
\sqrt{\rho^2+\theta^2y^2}c_{yy}(y,z)+\frac{\theta^2 y}{\sqrt{\rho^2+\theta^2y^2}}c_y(y,z)+\frac{\rho^2}{\sqrt{\rho^2+\theta^2y^2}}c_{zz}(y,z)+\lambda \frac{2\rho^2\theta^2}{(\rho^2+\theta^2y^2)^{3/2}}c(y,z)=0,\\[6pt]
c_y(1,z)-\lambda(1+\rho^{-2}\theta^2)^{-1}c(1,z)=0,\\[6pt]
c_y(-1,z)+\lambda (1+\rho^{-2}\theta^2)^{-1}c(-1,z)=0,\\[6pt]
c(y,0)=c(y,\rho)=0,
\end{array}
\right\}
\eeqn
for almost all $(y,z)\in [-1,1]\times[0,\rho]$. Here, the surface $\cS_\theta$ is identified with the rectangle $[-1,1]\times[0,\rho]$ via the parameterization \eqref{helicoid}.

Just as in the proof of Proposition \ref{prop4.1}, it is possible to reduce \eqref{hevp} to an eigenvalue problem involving a function of a single variable. 

\begin{proposition}\label{prop5.1}
The helicoid $\cS_\theta$ satisfies the second-variation stability condition if and only if for all $k\in\Nat$ the smallest eigenvalue $\lambda$ of the eigenvalue problem
\beqn\label{1dhevp}
\left .
\begin{array}{c}
-(\sqrt{\rho^2+\theta^2y^2}g')'+\frac{k^2\pi^2}{\sqrt{\rho^2+\theta^2y^2}}g=\lambda \frac{2\rho^2\theta^2}{(\rho^2+\theta^2y^2)^{3/2}}g\qquad {\rm on}\ [-1,1],\\[8pt]
g'(1)-\lambda(1+\rho^{-2}\theta^2)^{-1}g(1)=0,\\[6pt]
g'(-1)+\lambda (1+\rho^{-2}\theta^2)^{-1}g(-1)=0,
\end{array}
\right\}
\eeqn
is greater than or equal to $1$, where weak solutions to \eqref{1dhevp} are sought in $W^{1,2}([-1,1],\Real)$.
\end{proposition}

\begin{proof}
By Proposition \ref{propstcp2}, the second-variation stability condition is equivalent to the eigenvalue problem \eqref{evp}, which takes the form \eqref{hevp} if the surface is the helicoid $\cS_\theta$. Notice that every element $c$ in $W^{1,2}(\cS_\theta,\Real)$ satisfying $c=0$ on $\cW$ can be expressed as
\beqn
c(y,z)=\sum_{k\in\Nat}g_k(y)\sin(k\pi z/\rho)\qquad\text{for almost all}\ (y,z)\in[-1,1]\times[0,\rho],
\eeqn
where the functions $g_k$, $k\in\Nat$, are of class $W^{1,2}$ on $\cS_\theta$ and the sum converges in the associated norm. Substituting this expression for $c$ into \eqref{hevp} yields
\beqn
\left .
\begin{array}{l}
\sum\limits_{k\in\Nat}\big[ \sqrt{\rho^2+\theta^2y^2}g''_k(y)+\frac{\theta^2 y}{\sqrt{\rho^2+\theta^2y^2}}g_k'(y)-\frac{k^2 \pi^2}{\sqrt{\rho^2+\theta^2y^2}}g_k(y)\\
\qquad\qquad+\lambda \frac{2\rho^2\theta^2}{(\rho^2+\theta^2y^2)^{3/2}}g_k(y)\big]\sin(k\pi z/\rho)=0,\\[8pt]
\sum\limits_{k\in\Nat}\big[g'_k(1)-\lambda(1+\rho^{-2}\theta^2)^{-1}g_k(1)\big]\sin(k\pi z/\rho)=0,\\[8pt]
\sum\limits_{k\in\Nat}\big[g'_k(-1)+\lambda(1+\rho^{-2}\theta^2)^{-1}g_k(-1)\big]\sin(k\pi z/\rho)=0,
\end{array}
\right\}
\eeqn
for almost all $(y,z)\in[-1,1]\times[0,\rho]$.

Since the family of functions $z\mapsto\sin(k\pi z/\rho)$, $k\in\Nat$, forms a complete orthogonal basis for those functions in $W^{1,2}([0,\rho],\Real)$ with zero Dirichlet boundary condition, $g_k$ must obey \eqref{1dhevp} for each $k\in\Nat$. Hence, a solution to \eqref{hevp} with eigenvalue $\lambda$ induces a solution to \eqref{1dhevp} with eigenvalue $\lambda$ for each $k\in\Nat$.

Conversely, given a solution $g$ of \eqref{1dhevp} with eigenvalue $\lambda$ for some $k\in\Nat$, it is easy to show that the function
\beqn
c(y,z)=g(y)\sin(k\pi z/\rho)\qquad \rfaa (y,z)\in[-1,1]\times[0,\rho]
\eeqn
is a solution to \eqref{hevp} with eigenvalue $\lambda$. It follows that the smallest eigenvalue of \eqref{hevp} is equal to the smallest eigenvalue of \eqref{1dhevp}. This establishes the result.
\end{proof}

Notice that \eqref{1dhevp} has a structure similar to that of a Sturm--Liouville problem except that here the eigenvalue appears in both the differential equation and the boundary condition. Moreover, standard boot-strap type techniques can be used to show that any weak solution of \eqref{1dhevp} must be smooth.

From Proposition \ref{prop5.1}, it appears that an infinite number of eigenvalue problems must be considered, one for each $k\in\Nat$. It transpires that considering the problem for $k=1$ suffices. 

\begin{proposition}
The smallest eigenvalue of \eqref{1dhevp} occurs for $k=1$.
\end{proposition}

\begin{proof}
Start by multiplying \eqref{1dhevp}$_1$ by $g$, then integrate over $[-1,1]$, use integration by parts, and \eqref{1dhevp}$_{2,3}$ to obtain
\begin{align}
\int_{-1}^1[\sqrt{\rho^2+\theta^2y^2}g'(y)^2+&\frac{k^2\pi^2}{\sqrt{\rho^2+\theta^2y^2}}g(y)^2]\, \text{d}y\\
&=\lambda \Big[ \frac{\rho^2}{\sqrt{\rho^2+\theta^2}}(g^2(1)+g^2(-1))+\int_{-1}^{1}\frac{2\rho^2\theta^2}{(\rho^2+\theta^2y^2)^{3/2}}g(y)^2\,\text{d}y\Big].
\end{align}
It follows that smallest eigenvalue $\lambda$ of the eigenvalue problem \eqref{1dhevp} is given by
\beqn\label{mink}
\lambda=\inf_{g\in W^{1,2}([-1,1],\Real)}\frac{\int_{-1}^1[\sqrt{\rho^2+\theta^2y^2}g'(y)^2+\frac{k^2\pi^2}{\sqrt{\rho^2+\theta^2y^2}}g(y)^2]\, \text{d}y}{\frac{\rho^2}{\sqrt{\rho^2+\theta^2}}(g(1)^2+g(-1)^2)+\int_{-1}^{1}\frac{2\rho^2\theta^2}{(\rho^2+\theta^2y^2)^{3/2}}g(y)^2\,\text{d}y}.
\eeqn
It is clear from \eqref{mink} that $\lambda$ increases with increasing $k$. It follows that the minimum eigenvalue of \eqref{1dhevp} occurs for $k=1$.
\end{proof}

Consider the set
\beqn
\cK:=\{g\in W^{2,2}([-1,1])\ \big |\ \norm{g}_{W^{2,2}}=1\}
\eeqn
endowed with a topology by the $W^{1,2}$ norm. With this topology, $\cK$ is compact since $W^{2,2}([-1,1],\Real)$ is compactly embedded in $W^{1,2}([-1,1],\Real)$. Consider the mapping
\beqn
\bar\lambda:\cK\times(0,\infty)\times[0,\infty)\longrightarrow\Real
\eeqn
defined by
\beqn\label{barlambda}
\bar\lambda(g,\rho,\theta):=\frac{\int_{-1}^1[\sqrt{\rho^2+\theta^2y^2}g'(y)^2+\frac{\pi^2}{\sqrt{\rho^2+\theta^2y^2}}g(y)^2]\, \text{d}y}{\frac{\rho^2}{\sqrt{\rho^2+\theta^2}}(g(1)^2+g(-1)^2)+\int_{-1}^{1}\frac{2\rho^2\theta^2}{(\rho^2+\theta^2y^2)^{3/2}}g(y)^2\,\text{d}y}
\eeqn
for all $(g,\rho,\theta)\in\cK\times(0,\infty)\times[0,\infty)$. Let $\hat\lambda(\rho,\theta)$ be the smallest eigenvalue of the eigenvalue problem \eqref{1dhevp} as a function of $\rho$ and $\theta$. Since $W^{2,2}([-1,1],\Real)$ is dense in $W^{1,2}([-1,1],\Real)$, it follows from \eqref{mink} (with $k=1$) that 
\beqn\label{hatldef}
\hat\lambda(\rho,\theta)=\inf_{g\in\cK} \bar\lambda(g,\rho,\theta)\rfa (\rho,\theta)\in(0,\infty)\times[0,\infty).
\eeqn
It is straightforward to see that $\bar\lambda$ is continuous. A basic result from topology is that since $\cK$ is compact and $\bar\lambda$ is continuous, $\hat\lambda$ is also continuous.

It follows from the proof of Proposition \eqref{prop4.1} that
\beqn\label{loe}
\hat\lambda(\rho,0)=\frac{\pi}{\rho}\tanh\big(\frac{\pi}{\rho}\big)\rfa \rho\in(0,\infty).
\eeqn

The goal is to determine conditions under which $\hat\lambda$ is greater than $1$ or less than $1$ since, by Proposition \ref{prop5.1}, this determines when the surface $\cS_\theta$ is stable or unstable. To obtain a lower bound on $\hat\lambda$, the following two lemmas are useful.

\begin{lemma}\label{lemeq}
Fix $A,B,C>0$, and $D\geq 0$ and consider the minimization problem
\beqn\label{leminf}
\lambda:=\inf_{g\in W^{1,2}([-1,1],\Real)}\frac{\int_{-1}^1[Ag'(y)^2+Bg(y)^2]\, \text{d}y}{C(g(1)^2+g(-1)^2)+\int_{-1}^{1}Dg(y)^2\,\text{d}y}.
\eeqn
The infimum $\lambda$ satisfies
\begin{align}
\label{gen2}\lambda=&\frac{\sqrt{A}}{C}\sqrt{B-\lambda D}\tanh\sqrt{\frac{B-\lambda D}{A}}.
\end{align}
\end{lemma}

\begin{proof}
Using an argument similar to that in the proof of Proposition \ref{propstcp2} shows that \eqref{leminf} has a minimizer that satisfies
\beqn\label{lemevp}
\left .
\begin{array}{c}
Ag''-(B-\lambda D)g=0\qquad {\rm on}\ [-1,1],\\[6pt]
Ag'(1)-\lambda Cg(1)=0,\\[6pt]
Ag'(-1)+\lambda Cg(-1)=0,
\end{array}
\right\}
\eeqn

The parameters $B$, $D$, and $\lambda$ may satisfy either $B-\lambda D\ge0$ or $B-\lambda D<0$. If the latter alternative holds, the solution of \eqref{lemevp}$_1$ is easily shown to be of the form
\beqn
g(y):=E \cos(\mu y)+F\sin(\mu y)\rfa y\in [-1,1],
\eeqn
where $\mu:=\sqrt{\lambda D-B}$. The boundary conditions \eqref{lemevp}$_{2,3}$ yield
\begin{align}
E=0\qquad& \text{or}\qquad \lambda=\frac{A}{C}\mu\tan\mu,\\
\displaystyle
F=0\qquad& \text{or}\qquad \lambda=-\frac{A}{C}\mu\cot\mu.
\end{align}
Since $\lambda$ is the smallest eigenvalue,
\beqn
\lambda=\min\{-\frac{A}{C}\mu\cot\mu, \frac{A}{C}\mu\tan\mu\}\leq 0.
\eeqn
However, $\lambda\leq0$ is inconsistent with $B-\lambda D< 0$ and $B,C>0$. Hence, the inequality 
\beqn\label{gen1}
B-\lambda D\geq 0
\eeqn
must hold. Under this condition, the solution of \eqref{lemevp}$_1$ is of the form
\beqn
g(y):=E \cosh(\mu y)+F\sinh(\mu y)\rfa y\in [-1,1],
\eeqn
with $\mu:=\sqrt{B-\lambda D}$. The boundary conditions \eqref{lemevp}$_{2,3}$ yield
\begin{align}
E=0\qquad& \text{or}\qquad \lambda=\frac{A}{C}\mu\tanh\mu,\\
F=0\qquad& \text{or}\qquad \lambda=\frac{A}{C}\mu\coth\mu.
\end{align}
Since $\tanh\mu\leq\coth\mu$, the smallest eigenvalue is
\beqn
\lambda=\frac{A}{C}\mu\tanh\mu,
\eeqn
which is \eqref{gen2}. Notice that the right-hand side \eqref{gen2} is undefined when \eqref{gen1} fails to hold. Thus, it suffices to ensure that \eqref{gen2} holds.
\end{proof}

The second lemma, stated below, is easily establish and, hence, the proof will be omitted. It is stated so it can be referred to later.

\begin{lemma}\label{lemeq2}
Fix $A,B,\alpha,\beta>0$, and $C_L$, $C_U$, and $C$ such that $0\leq C_L\leq C\leq C_U$. Then,
\beqn
\frac{A+\alpha C}{B+\beta C}\geq\min\Big\{\frac{A+\alpha C_L}{B+\beta C_L},\frac{A+\alpha C_U}{B+\beta C_U}\Big\}.
\eeqn
\end{lemma}

The following result gives a condition for stability.

\begin{proposition}\label{proplow}
The surface $\cS_\theta$ is stable if
\beqn\label{lowcond}
1<\sqrt{\frac{\pi^2-2\theta^2}{\rho^2}}\tanh\sqrt{\frac{\pi^2-2\theta^2}{\rho^2}}.
\eeqn
\end{proposition}

\begin{proof}
Start by fixing $\rho\in(0,\infty)$ and $\theta\in[0,\infty)$ and notice that for all $g\in\cK$,
\begin{align}
\bar\lambda(g,\rho,\theta)&=\frac{\int_{-1}^1[\sqrt{(\rho^2+\theta^2y^2)(\rho^2+\theta^2)}g'(y)^2+\pi^2\sqrt{\frac{\rho^2+\theta^2}{\rho^2+\theta^2y^2}}g(y)^2]\, \text{d}y}{\rho^2(g(1)^2+g(-1)^2)+\int_{-1}^{1}\frac{2\rho^2\theta^2}{\rho^2+\theta^2y^2}\sqrt{\frac{\rho^2+\theta^2}{\rho^2+\theta^2y^2}}g(y)^2\,\text{d}y}\\
\label{lowbd}&\geq \frac{\int_{-1}^1[\rho^2g'(y)^2+\pi^2\sqrt{\frac{\rho^2+\theta^2}{\rho^2+\theta^2y^2}}g(y)^2]\, \text{d}y}{\rho^2(g(1)^2+g(-1)^2)+\int_{-1}^{1}2\theta^2\sqrt{\frac{\rho^2+\theta^2}{\rho^2+\theta^2y^2}}g(y)^2\,\text{d}y}
\end{align}
and
\beqn
\int_{-1}^1g(y)^2\,\text{d}y\leq \int_{-1}^1\sqrt{\frac{\rho^2+\theta^2}{\rho^2+\theta^2y^2}}g(y)^2\,\text{d}y\leq\int_{-1}^1\frac{1}{\rho}\sqrt{\rho^2+\theta^2}g(y)^2\,\text{d}y.
\eeqn

Taking the infimum of both sides of \eqref{lowbd} over $g\in\cK$ and using Lemma \ref{lemeq} and Lemma \ref{lemeq2} yields
\beqn\label{hlambdalb}
\hat\lambda(\rho,\theta)\geq\min\{\lambda_1,\lambda_2\},
\eeqn
where $\lambda_1$ must satisfy
\beqn
\label{low2}\lambda_1=\sqrt{\frac{\pi^2-2\lambda_1\theta^2}{\rho^2}}\tanh\sqrt{\frac{\pi^2-2\lambda_1\theta^2}{\rho^2}}
\eeqn
and $\lambda_2$ must satisfy
\beqn
\label{low2p}\lambda_2=\frac{1}{\rho^{3/2}}\sqrt{\pi^2\sqrt{\rho^2+\theta^2}-2\lambda_2\theta^2\sqrt{\rho^2+\theta^2}}\tanh\sqrt{\frac{\pi^2\sqrt{\rho^2+\theta^2}-2\lambda_2\theta\sqrt{\rho^2+\theta^2}}{\rho^3}}.
\eeqn
Using the relation $\sqrt{\rho^2+\theta^2}\geq\rho$, it follows from \eqref{low2} and \eqref{low2p} that $\lambda_2\geq\lambda_1$. Thus, from \eqref{hlambdalb} and Proposition \ref{prop5.1}, the surface $\cS_\theta$ is stable if $\lambda_1>1$. 

For the particular choice $\lambda_1=1$, \eqref{low2} specializes to
\beqn
\label{low2.1}1=\sqrt{\frac{\pi^2-2\theta^2}{\rho^2}}\tanh\sqrt{\frac{\pi^2-2\theta^2}{\rho^2}}.
\eeqn
It follows that the condition $\lambda_1>1$ holds if either 
\beqn
\label{low2.1.1}1>\sqrt{\frac{\pi^2-2\theta^2}{\rho^2}}\tanh\sqrt{\frac{\pi^2-2\theta^2}{\rho^2}}
\eeqn
or
\beqn
\label{low2.1.2}1<\sqrt{\frac{\pi^2-2\theta^2}{\rho^2}}\tanh\sqrt{\frac{\pi^2-2\theta^2}{\rho^2}}.
\eeqn
To determine which of these last two inequalities ensures that $\lambda_1>1$, start by considering \eqref{low2} with $\theta=0$ and $\rho=1$ to obtain
\beqn
\lambda_1=\pi\tanh \pi>1.
\eeqn
Using the same values of $\rho$ and $\theta$ in the right-hand side of \eqref{low2.1} yields a value strictly greater than $1$ and, hence, \eqref{low2.1.2}, which is equivalent to \eqref{lowcond}, gives $\lambda_1>1$.
\end{proof}

To obtain a condition for instability, first notice that for any $\bar g\in \cK$,
\beqn
\hat\lambda(\rho,\theta)\leq \bar\lambda(\bar g,\rho,\theta)\rfa (\rho,\theta)\in(0,\infty)\times[0,\infty).
\eeqn
Thus, if $\bar\lambda(\bar g,\rho,\theta)<1$, by Proposition \ref{prop5.1}, the helicoid $\cS_\theta$ would be unstable. Consider the function
\beqn\label{barg}
\bar g(y):=A\cosh\big(\frac{\pi}{\sqrt{\rho^2+\theta^2y^2}}\big)\frac{\rho}{\sqrt{\rho^2+\theta^2y^2}}\rfa y\in[-1,1],
\eeqn
where $A$ is chosen so that the $W^{2,2}$ norm of $\bar g$ is $1$ and, hence, $\bar g\in\cK$. There are two reasons for choosing this form for $\bar g$. First, for $\theta=0$, $\bar g$ reduces to $g_1$ in \eqref{g1}, which is the minimizer for $\hat\lambda$ when $\theta=0$. Second, the quantity $\sqrt{\rho^2+\theta^2y^2}$ which appears in \eqref{barg}, also appears in several places in \eqref{barlambda}.

\section{Discussion}

The inequalities 
\beqn\label{upcond2}
1<\sqrt{\frac{\pi^2-2\theta^2}{\rho^2}}\tanh\sqrt{\frac{\pi^2-2\theta^2}{\rho^2}}
\eeqn
and
\beqn\label{lowcond2}
1>\bar\lambda(\bar g,\rho,\theta),
\eeqn
which respectively supply conditions for stability and instability, can be plotted. In Figure 3, the blue region consists of those points satisfying the stability condition \eqref{upcond2} and the red region consists of those points satisfying the instability condition \eqref{lowcond2}. The results suggest that there is a curve in the $(\rho,\theta)$-plane that separates the stable helicoids from the unstable helicoids and that this curve emanates from the vertical axis, decreasing for increasing $\rho$ until it intersects the $\rho\hspace{.01in}$-axis at the value of $\rho$ that satisfies $\frac{\pi}{\rho}\tanh(\frac{\pi}{\rho}\big)=1$. This curve must necessarily fall between the red and blue regions in Figure 3. However, it is important to realize that the results of the previous section do not establish the existence of such a curve.  It is possible that the stable and unstable regions are separated by a nonmonotonically decreasing curve that intersects the $\rho\hspace{.01in}$-axis. Moreover, the results do not even guarantee that the stable and unstable regions are separated by a curve.

The points in Figure 3 marked with circles and squares correspond to marginally unstable helical surfaces obtained experimentally and numerically by Cox \& Jones \cite{CJ}. Each such helicoid becomes unstable if the ratio $\rho$ is increased. While the points marked with circles represent experimental data with measurement errors of $\pm 0.15$ in $\rho$ and $\pm0.17$ in $\theta$, those marked with squares represent numerical solutions obtained using Surface Evolver with errors of $\pm0.01$ in $\rho$ and $\pm0.00017$ in $\theta$. Taking into account these errors, the results of the previous section are consistent with both the experimental and numerical results of Cox \& Jones \cite{CJ}. Moreover, it appears that the inequality \eqref{lowcond2} determining the unstable region is close to optimal, while the inequality \eqref{upcond2} determining the stable region leaves room for improvement.

\begin{figure}[t]\label{figstplot}
\centering
\begin{picture}(450,230)
\put (110,0)  {\includegraphics[width=3in]{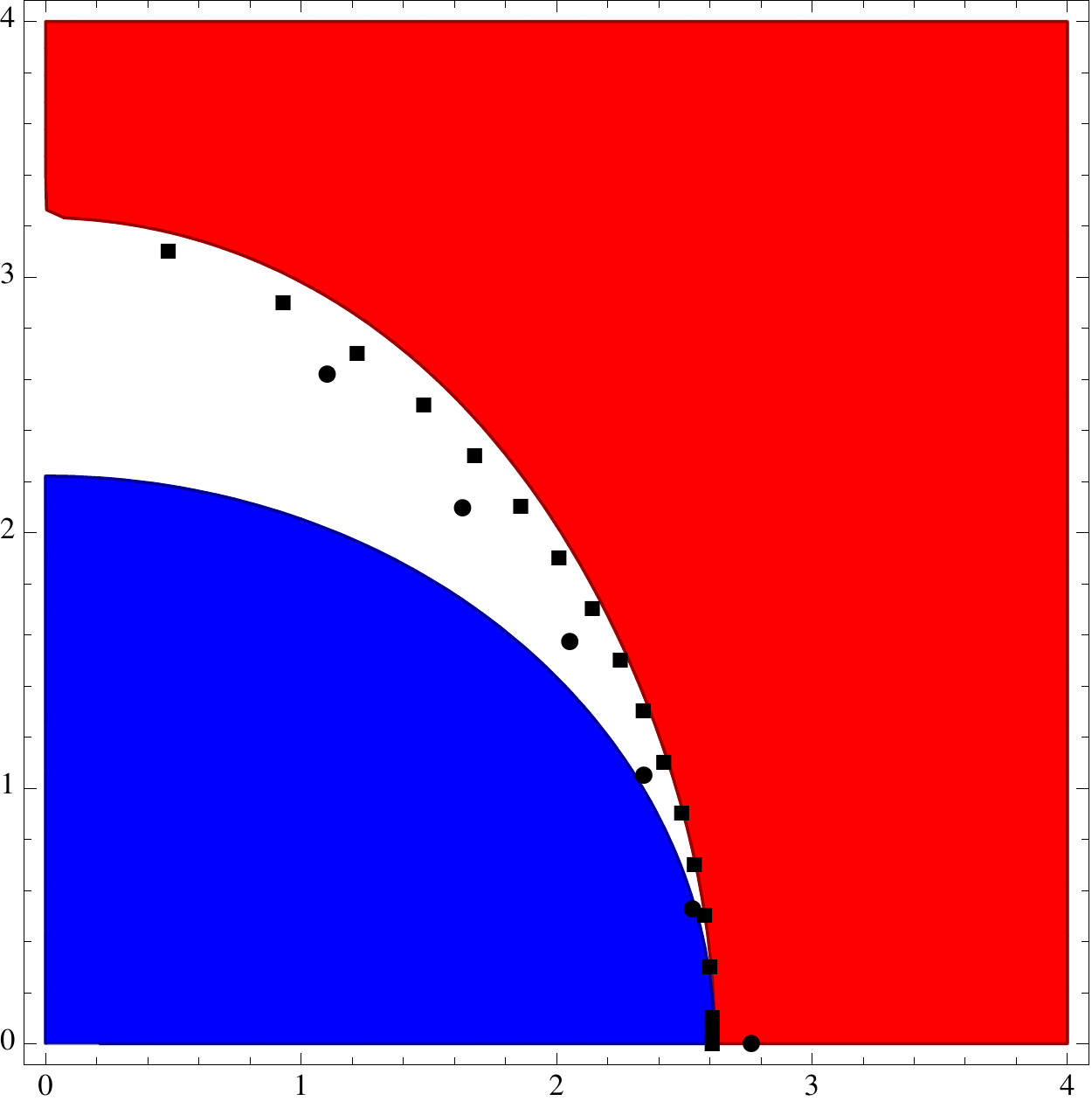}}
\end{picture}
\put(-350,112){$\theta$}
\put(-233,-10){$\rho$}
\caption{A depiction of the stability results for the helical surface $\cS_\theta$. While the points in the red region satisfy \eqref{upcond2} and correspond to unstable surfaces, the points in the blue region satisfy \eqref{lowcond} and correspond to stable surfaces. The red and blue regions meet on the horizontal axis at the value of $\rho$ that satisfies $\frac{\pi}{\rho}\tanh(\frac{\pi}{\rho}\big)=1$. The analysis in this section is inconclusive regarding the surfaces corresponding to the points in the white region. The points in the white region correspond to marginally stable helical surfaces obtained Cox \& Jones~\cite{CJ}; while the circles represent data from experiments with measurement errors of $\pm0.15$ in $\rho$ and $\pm0.17$ in $\theta$, the squares represent numerical solutions obtained using Surface Evolver with errors of $\pm0.01$ in $\rho$ and $\pm0.00017$ in $\theta$.}
\end{figure}

Although the precise boundary between the stable and unstable regions remains undetermined, a number of interesting facts can be deduced from Figure 3. In particular, since $\rho=L/R$, large values of $\rho$ correspond to long, thin cylinders and small values of $\rho$ correspond to short, thick cylinders, it is evident that
\begin{enumerate}

\item For sufficiently long, thin cylinders ($2.62\lessapprox \rho$), no helicoid is stable.

\item Regardless of the dimensions of the cylinder, no helicoid that contains a little over half of a rotation is stable.

\item As long as the cylinder is not overly long and thin ($\rho\lessapprox 2.62$) and $\theta_0$ is not too large ($\theta_0\lessapprox2.2$), there exist nonflat stable helicoids.

\end{enumerate}

\subsection*{Acknowledgement} 

We thank Gantumur Tsogtgerel, David Shirokoff, and Marco Veneroni for very fruitful discussions, Aisa Bira for assistance with Figures 1 and 2, and Russell Todres for helpful input.

\end{document}